\documentclass[11pt]{amsart}
\usepackage{amsthm}
\theoremstyle{plain}
\newtheorem{theorem}{Theorem}[section]
\newtheorem{lemma}[theorem]{Lemma}
\newtheorem{corollary}[theorem]{Corollary}
\newtheorem*{theorem*}{Theorem}
\theoremstyle{definition}
\newtheorem{definition}[theorem]{Definition}
\newtheorem{remark}[theorem]{Remark}
\newtheorem{example}[theorem]{Example}
\usepackage{amssymb, tikz}
\usepackage[shortlabels]{enumitem}

\newcommand{\Hom}{\mathrm{Hom}}
\newcommand{\dual}{\vee}
\newcommand{\trop}[1]{\mathrm{trop}\left(#1\right)}
\newcommand{\tropim}[2]{\mathrm{tropim}_{#1}\left(#2\right)}
\newcommand{\tropdet}[1]{\mathrm{tdet}\left(#1\right)}
\newcommand{\puiseux}[1]{ \kappa\{t^{#1}\}} 
\newcommand{\val}{\mathrm{val}}
\DeclareMathOperator{\rank}{rk}

\newcommand{\stsum}{+_{st}}
\newcommand{\stint}{\cap_{st}}

\newcommand{\contract}{/}
\newcommand{\delete}{\!\setminus\!}

\newcommand{\B}{\mathbb B}
\newcommand{\T}{\mathbb T}
\newcommand{\R}{\mathbb R}

\newcommand{\linext}[3]{#1 +_{#2} #3}

\usepackage{color}

\title{The image of a tropical linear space}
\author{Joshua Mundinger}
\date{\today}

\usepackage[foot]{amsaddr}
\address{Department of Mathematics, University of Chicago}
\email{mundinger@uchicago.edu}

\begin{document}

\begin{abstract}
	Given a tropical linear space $L \subseteq \T^n$ and a matrix $A \in \T^{m \times n}$, the image $AL$ of $L$ under $A$ is typically not a tropical linear space.
	We introduce a tropical linear space $\tropim{A}{L}$, the tropical image, containing $AL$.
	We show under mild hypotheses that $\tropim{A}{L}$ is realizable if $L$ is
	and apply the tropical image to construct the stable sum of two tropical linear spaces without a disjoint pair of bases.
\end{abstract}

\maketitle
\section{Introduction}

In tropical geometry, a single polynomial defines a tropical hypersurface,
but not every intersection of hypersurfaces is a tropical variety.
Morphisms of tropical varieties carry similar difficulties.
For instance, the image of a tropicalized morphism (a tuple of tropicalized polynomials) typically is not a tropical variety, let alone equal to the tropicalization of the image. 
Can the image of a tropical morphism be extended to a tropical variety?

In this paper, we give an algebraic treatment of this question in the case of a linear map on a tropical linear space,
drawing on the algebraic treatment of tropical linear spaces of \cite{frenk13} and the construction of Stiefel tropical linear spaces in \cite{fink15}.
Given a matrix $A$ and a tropical linear space $L$,
we introduce a tropical linear space $\tropim{A}{L}$, called the \emph{tropical image}. 
The tropical Pl\"ucker coordinates of the tropical image are determined by the minors of $A$ and the coordinates of $L$,
analogously to the classical situation.
\begin{theorem*}
	The tropical image $\tropim{A}{L}$ enjoys the following properties:
	\begin{itemize}
		\item it contains the set-theoretic image $AL$;
		\item its rank is at most the rank of $L$;
		\item the underlying matroid of $\tropim{A}{L}$ is induced from the underlying matroid of $L$ via the bipartite graph underlying $A$;
		\item if $A$ has a non-zero maximal minor, then $\tropim{A}{S^E}$ is the Stiefel troipcal linear space associated to $A^T$;
		\item if $L = \trop\Lambda$ is realizable over a sufficiently large field, then $\tropim{A}{L}$ is realizable
			and is equal to the tropicalization of an image of $\Lambda$;
		\item the stable sum $L \stsum L'$ of tropical linear spaces is the tropical image of $L \oplus L'$ under the matrix $\begin{bmatrix} I & I \end{bmatrix}$ representing addition.
	\end{itemize}
\end{theorem*}

The tropical image allows us to unify and generalize constructions appearing in the literature, for instance by extending Stiefel tropical linear spaces to all matrices, or the stable sum to all pairs of linear spaces.
The tropical image is constructed by first constructing a linear space containing the set-theoretic graph of a linear function, and then projecting.
The tropical graph coincides with an iterated tropical modification. 
By giving an algebraic treatment and connecting to the matroid literature the linear case, we are able to give examples of unexpected phenomena in tropical geometry, for instance an example of a tropical modification $L' \to L$ where $L$ and $L'$ are tropical linear spaces but the function corresponding to the modification is \textit{non}-linear.
\\
\paragraph{\emph{Acknowledgments.} 
The author thanks Noah Giansiracusa for invaluable discussion and guidance.
}
\section{Preliminaries}

\subsection{Modules over idempotent semifields}

A \emph{semiring} is a set with two binary operations that satisfy the axioms of a (commutative) ring except for the existence of additive inverses.
The additive identity will be denoted $0$, and the multiplicative identity $1$.
A \emph{semifield} is a semiring where every nonzero element has a multiplicative inverse.
A semiring is \emph{idempotent} if and only if $1+1=1$.
Idempotent semirings have a canonical ordering given by $a \leq b$ if and only if $a+b =b$;
if this ordering is total, the semiring is said to be \emph{totally ordered}.
For tropical geometry, the main example of a totally ordered semifield is the tropical semifield $\T = \R \cup \{-\infty\}$ with the operations of maximum and the usual real addition.

A module $N$ over a semifield $S$ is an abelian monoid with a homomorphism $S \to End(N)$.
The dual of an $S$-module $N$ is $N^\dual := \Hom(N,S)$. 
If $S$ is an idempotent semifield and $E = \{e_1,e_2,\ldots,e_n\}$ is a finite set, the free $S$-module with basis $E$ will be denoted $S^E$. 
Unlike the situation for fields, $S^E$ has a unique basis up to permutation and scaling \cite[Proposition 2.2.2]{giansiracusa17}.
The dual $(S^E)^\dual$ is also free;
the dual basis to $E$ will usually be denoted $\{x_1,x_2,\ldots,x_n\}$
and satisfies
\begin{equation*}
	x_i(e_j) = \begin{cases} 1 & i = j \\ 0 & i \neq j \end{cases}.
\end{equation*}	
\begin{remark}
	Totally ordered idempotent semifields are equivalent to totally ordered abelian groups,
	in that $S \mapsto S^\times$ is an equivalence of categories.
	The language of semifields is used to facilitate module theory,
	as well as to bring out structural analogies to classical algebraic geometry.
\end{remark}

\begin{definition}(cf. \cite[\S 2.3]{giansiracusa17})
	Let $S$ be a totally ordered idempotent semifield,
	and let $f \in (S^E)^\dual$ be a linear form.
	The \emph{tropical hyperplane} defined by $f$
	is the submodule 
	\begin{equation*}
		\left\{v = \sum_{i=1}^n v_ie_i \in S^E: f(v) = f\left(\sum_{i \neq j} v_ie_i\right) \text{ for all } j \right\}.
	\end{equation*}
\end{definition}
This definition uses the formulation of the algebraic \textit{bend relations} introduced in \cite{giansiracusa16}. If $f = \sum_{i=1}^n f_ix_i$, then the condition that $v = \sum_{i=1}^n v_ie_i$ is in the tropical hyperplane defined by $f$ is equivalent to saying that the maximum term of $\{f_iv_i\}_{i=1}^n$ is achieved twice.

\subsection{Tropical linear spaces}
Tropical linear spaces, introduced by Speyer \cite{speyer08}, are defined by tropical Pl\"ucker coordinates,
which satisfy the tropicalization of the classical quadratic Pl\"ucker relations.
Here and below, $\binom{E}{d}$ will denote the set of subsets of $E$ of size $d$.
\begin{definition}
	Let $S$ be a totally ordered idempotent semifield.
	A \emph{tropical Pl\"ucker vector} of rank $d$ on ground set $E$
	is a nonzero vector $w \in S^{\binom{E}{d}}$ satisfying the \emph{tropical Pl\"ucker relations}:
	for any $J \in \binom{E}{d+1}$ and $K \in \binom{E}{d-1}$,
	\begin{equation*}
		\sum_{i \in J - K} w_{J-i}w_{K+i} = \sum_{i \in J-K, i \neq j} w_{J-i}w_{K+i}
	\end{equation*}
	for all $j \in J-K$.
\end{definition}
Again we use the formulation of the tropical Pl\"ucker relations via bend relations as in \cite{giansiracusa17}.
The coordinates of a tropical Pl\"ucker vector are also known as tropical Pl\"ucker coordinates.
\begin{definition}
	Let $S$ be a totally ordered idempotent semifield,
	and let $w \in S^{\binom{E}{d}}$ be a tropical Pl\"ucker vector.
	Then the \emph{tropical linear space} associated to $w$
	is the intersection of the tropical hyperplanes defined by 
	\begin{equation*}
		\sum_{i \in J} w_{J-i}x_i	
	\end{equation*}
	over all $J \in \binom{E}{d+1}$.
	The tropical linear space associated to $w$ is denoted $L_w$.
\end{definition}
Tropical linear spaces determine their tropical Pl\"ucker coordinates up to a scalar in $S^\times$ (first proven over $\T$ in \cite{speyer08} and \cite{maclagan15}, and for any $S$ in \cite[\S 6.2]{giansiracusa17}).

Baker and Bowler's framework of matroids over hyperfields, and more generally, matroids over tracts, 
has provided a more general framework for tropical Pl\"ucker vectors
and tropical linear spaces \cite{baker16}.
A totally ordered idempotent semifield $S$ defines a hyperaddition operator $\boxplus$ by 
\begin{equation*}
	a \boxplus b = \begin{cases} a+b & a \neq b \\ \{x: x \leq a\} & a = b \end{cases}.
\end{equation*}
Replacing addition with hyperaddition makes $S$ a hyperfield,
and matroids over this hyperfield in the sense of \cite{baker16} are equivalent to tropical linear spaces;
in terms of hyperfields, $L_w$ is the set of covectors in the sense of \cite{anderson18}.

Unlike matroids over tracts, tropical linear spaces are submodules of the ambient free module, since tropical hyperplanes are.
The following lemma gives a generating set for a tropical linear space,
known as the \emph{valuated cocircuits}.
The hypothesis of totally ordered coefficients is essential.
\begin{lemma}{\cite[Proposition 4.1.9]{frenk13}} \label{lemma: generated-by-cocircuits}
	Let $S$ be a totally ordered idempotent semifield,
	and let $w \in S^{\binom{E}{d}}$ be a tropical Pl\"ucker vector.
	Then $L_w$ is generated as an $S$-module by
	\begin{equation*}
		\left\{ \sum_{i \in E -K} w_{K+i}e_i : K \in \binom{E}{d-1}\right\}.
	\end{equation*}
\end{lemma}

\subsection{Subspaces and the tropical incidence relations}

For two tropical Pl\"ucker vectors $w$ and $z$ on ground set $E$, when is $L_w \subseteq L_z$?
This occurs if and only if $w$ and $z$ satisfy certain relations, which turn out to be the tropical analogue of incidence relations.
These relations first appeared in \cite{haque12} over $\T$ but hold over any idempotent semifield.
\begin{lemma} \label{lemma: tropical-incidence-relations}
	Let $S$ be a totally ordered idempotent semifield, 
	and let $w \in S^{\binom{E}{d}}$ and $z \in S^{\binom{E}{e}}$ be tropical Pl\"ucker vectors.
	Then $L_w \subseteq L_{z}$ if and only if 
	\begin{equation} \label{equation: tropical-incidence-relations}
		\sum_{i \in A - B} z_{A-i}w_{B+i} = \sum_{i \in A-B, i \neq j} z_{A-i}w_{B+i}
	\end{equation}
	for all $A \in \binom{E}{e+1}$, $B \in \binom{E}{d-1}$, and $j \in A-B$.
\end{lemma}
\begin{proof}
	The equations \eqref{equation: tropical-incidence-relations} are exactly requiring
	that for any $B \in \binom{E}{d-1}$ and $A \in \binom{E}{e+1}$,
	the valuated cocircuit $\sum_{i \in E - B} w_{B+i}e_i$
	of $w$ is contained in 
	the hyperplane defined by $\sum_{i \in A} z_{A -i} x_i$,
	one of the hyperplanes defining $L_z$.
	Hence, the equations \eqref{equation: tropical-incidence-relations} hold if and only if
	all the valuated cocircuits of $w$ are contained in $L_{z}$.
	By Lemma \ref{lemma: generated-by-cocircuits},
	$L_w$ is generated as an $S$-module by its valuated cocircuits,
	and tropical hyperplanes are $S$-modules.
\end{proof}

The equations \eqref{equation: tropical-incidence-relations} are known as the \emph{tropical incidence relations}.

\subsection{Exterior algebra, and operations on Pl\"ucker vectors}


In \cite{giansiracusa17}, an analog of the exterior algebra for idempotent semirings was introduced,
providing a convenient framework for the algebra of tropical Pl\"ucker vectors.
If $S$ is an idempotent semifield, the \emph{tropical Grassmann algebra} or \emph{exterior algebra} on $S^E$
is the $S$-algebra quotient of the symmetric algebra on $S^E$ (defined in the usual way) by the relations $e_i^2 \sim 0$.
The exterior algebra is denoted by $\bigwedge S^E$, and its multiplication will be denoted by $\wedge$. 
If $I = \{i_1,\ldots, i_d\} \subset E$, let $e_I := e_{i_1} \wedge e_{i_2} \wedge\cdots \wedge e_{i_d}$.
The $d$th graded piece $\bigwedge^d S^E$ is free with basis 
$\left\{e_I\mid I \subseteq E, |I|=d\right\}.$
Hence, we may consider a tropical Pl\"ucker vector of rank $d$ as a vector in $\bigwedge^d S^E$.

There are three operations of interest on the exterior algebra that preserve the tropical Pl\"ucker relations.
The first is the multiplication $\wedge$.
The second is the Hodge star $\star: \bigwedge^d S^E \to \bigwedge^{|E|-d} (S^E)^\dual$,
which maps $e_I$ to $x_{E-I}$.
Combining these gives the third operation 
$\cdot: \bigwedge^d S^E \times \bigwedge^{d'} S^n \to \bigwedge^{d+d' - n} S^E$
defined by $w \cdot w' = \star(\star w \wedge \star w')$.

\begin{lemma}\label{lemma: plucker-ops} 
	Let $S$ be a totally ordered idempotent semifield.
	\begin{enumerate}[(a)]
		\item{If $w \in \bigwedge^d S^E$ and $w' \in \bigwedge^{d'} S^E$ are tropical Pl\"ucker vectors
			such that $w \wedge w' \neq 0$, then $w \wedge w'$ is a tropical Pl\"ucker vector.}
		\item{If $w \in \bigwedge^d S^E$ is a tropical Pl\"ucker vector, then so is $\star w$.}
		\item{If $w \in \bigwedge^d S^E$ and $w' \in \bigwedge^{d'} S^E$ are tropical Pl\"ucker vectors
		such that $w \cdot w' \neq 0$, then $w \cdot w'$ is a tropical Pl\"ucker vector.}
	\end{enumerate}
\end{lemma}
\begin{proof}
	(a) is \cite[Proposition 5.1.2]{giansiracusa17}. (b) can be checked from the definition, and (c) follows from (a) and (b).
\end{proof}

\begin{remark}
	For $\varphi \in (S^E)^\dual$, the tropical linear space $L_{\star \varphi}$ is the tropical hyperplane defined by $\varphi$,
	for 
	\begin{equation*}
		\sum_{i \in E} (\star \varphi)_{E-i} x_i = \sum_{i \in E} \varphi(e_i) x_i = \varphi,
	\end{equation*}
	and hence $\sum_{i=1}^n c_ie_i$ is in the hyperplane defined by $\varphi$ if and only if it is in the tropical linear space with tropical Pl\"ucker coordinates $\star\varphi$.
\end{remark}

All of these operations, when they result in a nonzero vector, have geometric interpretations: 
$L_{w \wedge w'}$ is the stable sum of $L_w$ and $L_{w'}$ \cite{fink15},
$L_{\star w}$ is the tropical orthogonal dual of $L_w$, and $L_{w \cdot w'}$ is the stable intersection of $L_w$ and $L_{w'}$ \cite{speyer08}.
Speyer introduced the stable intersection over $\T$ when all tropical Pl\"ucker coordinates are non-zero
\cite{speyer08} and gave a geometric interpretation in terms of polyhedral complexes.

Geometrically, we would expect the orthogonal dual reverses inclusions, while stable sum and stable intersection should preserve inclusions.
For tropical linear spaces over $\T$, this follows immediately from the polyhedral geometry of the operations.
Here, we provide an algebraic proof that works over any totally ordered idempotent semifield.
\begin{lemma} \label{lemma: wedge-is-monotonic}
	Let $S$ be a totally ordered idempotent semifield.
	\begin{enumerate}[(a)]
		\item{If $w \in \bigwedge^d S^E$ and $z \in \bigwedge^{e} S^E$ are tropical Pl\"ucker vectors
			and $L_w \subseteq L_{z}$, then $L_{\star z} \subseteq L_{\star w}$.}
		\item{If $w \in \bigwedge^d S^E, w' \in \bigwedge^{d'}S^E, z \in \bigwedge^e S^E,$ and $z' \in \bigwedge^{e'} S^E$
			are tropical Pl\"ucker vectors such that $z \wedge z' \neq 0$,
			\begin{equation*}
				L_w \subseteq L_z \qquad \text{and} \qquad L_{w'} \subseteq L_{z'},
			\end{equation*}
			then $w \wedge w' \neq 0$, and $L_{w \wedge w'} \subseteq L_{z \wedge z'}$.
		}
		\item{If $w \in \bigwedge^d S^E, w' \in \bigwedge^{d'}S^E, z \in \bigwedge^e S^E,$ and $z' \in \bigwedge^{e'} S^E$
			are tropical Pl\"ucker vectors such that $w \cdot w' \neq 0$,
			\begin{equation*}
				L_w \subseteq L_z \qquad \text{and} \qquad L_{w'} \subseteq L_{z'},
			\end{equation*}
			then $z \cdot z' \neq 0$, and $L_{w \cdot w'} \subseteq L_{z \cdot z'}$.
		}
	\end{enumerate}
\end{lemma}
\begin{proof}
	(a) follows from Lemma \ref{lemma: tropical-incidence-relations}
	by observing that $w$ and $z$ satisfy the tropical incidence relations
	if and only if $\star z$ and $\star w$ do, since
	\begin{equation*}
		\sum_{i \in A - B} z_{A-i}w_{B+i} = \sum_{i \in B^c - A^c} (\star w)_{B^c -i} (\star z)_{A^c + i}
	\end{equation*}
	and similarly dropping a term from the right-hand side is exactly dropping a term from the left-hand side.
	Alternatively, (a) follows from the characterization of tropical orthogonal duality of \cite[Corollary 4.4.4]{giansiracusa17}.

	The proof of (b) is a generalization of \cite[Proposition 5.1.2]{giansiracusa17},
	making use of Lemma \ref{lemma: tropical-incidence-relations}.
	That $w \wedge w' \neq 0$ follows from \cite[Theorem 6.5]{crowley17}.
	Now suppose that $A \in \binom{E}{e+e' + 1}$ and $B \in \binom{E}{d+d' - 1}$.
	Then 
	\begin{equation} \label{equation: wedge-incidence}
		\sum_{i \in A - B} (z \wedge z')_{A -i} (w \wedge w')_{B+i} 
		= \sum_{i \in A - B} \sum_{\substack{C \sqcup C' = A -i \\ D \sqcup D' = B+i}} z_Cz'_{C'}w_Dw'_{D'}.
	\end{equation}
	Collecting terms with $i \in D$ and $i \in D'$ separately, and regrouping, gives the equal expression
	\begin{equation*}
		\sum_{\substack{J \sqcup C' = A \\ K \sqcup D' = B}} (z'_{C'}w'_{D'}) \! \left(\sum_{i\in J -K} z_{J-i}w_{K+i}\right)
		+ \sum_{\substack{C \sqcup J' = A \\ D \sqcup K' = B}} (z_Cw_D)\! \left(\sum_{i \in J'-K'} z'_{J'-i}w'_{K'+i}\right)
	\end{equation*}
	Further, the terms with fixed $j \in A-B$ in \eqref{equation: wedge-incidence}
	are exactly the terms with $j \in J-K$ or $j \in J'-K'$ in the above sums.
	Since $L_w \subseteq L_z$ and $L_{w'} \subseteq L_{z'}$, 
	Lemma \ref{lemma: tropical-incidence-relations} shows that the sums above are equal to the same sums with $j$ dropped.
	Hence, the tropical incidence relations for $w \wedge w'$ and $z \wedge z'$ are satisfied,
	so by Lemma \ref{lemma: tropical-incidence-relations} again, $L_{w \wedge w'} \subseteq L_{z \wedge z'}$.

	(c) follows from (a) and (b).
\end{proof}

\subsection{Matroids}
A \emph{matroid} $M$ on ground set $E$ (assumed to be finite) is a collection of \emph{bases}, subsets of $E$, satisfying the strong exchange axiom: if $I$ and $J$ are bases of $M$ and $i \in I - J$, then there exists $j \in J - I$ such that $I -i +j $ and $J - j + i$ are bases of $M$.
Matroids may be defined in a number of cryptomorphic ways, and standard matroid terminology will be used in various remarks and examples in this paper; in particular, at times we will recall the deletion $M\delete F$ and contraction $M \contract F$ for $F \subseteq E$.
See \cite{oxley92} for more details.

Matroids are equivalent to tropical linear spaces over the two-element idempotent semifield $\B = \{0,1\}$, 
with additive identity $0$ and multiplicative identity $1$:
a vector $w \in \bigwedge \B^E$ is a tropical Pl\"ucker vector if and only if it is the indicator vector of the bases of a matroid. 
See \cite{crowley17} for more on this correspondence. 
Indeed, any tropical linear space $L_w$ over an idempotent semifield $S$ defines a matroid,
by taking the bases to be those coordinates of $w$ that are non-zero.
This matroid is called the \emph{underlying matroid} of $L_w$. 
Another way to understand this construction from a tropical linear space $L_w$
is to take the push-forward under the canonical semiring homomorphism $S \to \B$ (i.e.\! the image under applying $S \to \B$ coordinatewise)
and then use the above correspondence.
See \cite[\S 4]{baker16} for more details on push-forwards in the context of matroids over hyperfields.

\subsection{Minors of tropical linear spaces}

Frenk defined the minors of a tropical linear space in \cite{frenk13}, 
following the definition of the minors of a valuated matroid given by Dress and Wenzel \cite{dress92}.
Minors of matroids over hyperfields were defined by Baker and Bowler \cite{baker16}.
The following is essentially \cite[Lemma 6.4]{baker16}, translated to the language of totally ordered idempotent semifields. 
It shows that coordinate subspaces and projections of tropical linear spaces are again tropical linear spaces.
\begin{lemma} \label{lemma: minor-plucker-coordinates}
	Let $S$ be a totally ordered idempotent semifield, and let $w \in \bigwedge^d S^E$ be a tropical Pl\"ucker vector.
	Let $M$ be the underlying matroid of $w$.
	Let $F \subseteq E$ be fixed, and let $J \subseteq E-F$.
	Define $z \in S^{\binom{F}{d-|J|}}$ by
	\[z_I = w_{J \cup I}.\]
	\begin{enumerate}[(a)]
		\item{If $|J| < \rank M\contract F$ or $|J| > \rank M\delete F$, then $z$ is zero.
		}
		\item{if $|J| = \rank M\contract F$, then $z$ is either zero 
			or the Pl\"ucker vector of $\pi_F(L_w)$,
			where $\pi_F: S^E \to S^F$ is the coordinate projection map;
		}
		\item{if $|J| = \rank M\delete F$, then $z$ is either zero
			or the Pl\"ucker vector of $S^F \cap L_w$.
		}
	\end{enumerate}
\end{lemma}
\begin{proof}
	(a)	If $|J| < \rank M\contract F$ or $|J| > \rank M\delete F$,
	then there is no $I \subseteq F$ such that $J \cup I$ is a basis for $M$.

	(b) If $|J| = \rank M\contract F$ and there is $I \subseteq F$ such that $w_{I\cup J} \neq 0$,
	then $|I| = \rank M \delete (E-F)$, so $J$ is a basis for $M \contract F$.
	Then apply \cite[Lemma 6.4]{baker16}.

	(c) This follows from the dual argument to (b) and the dual part of \cite[Lemma 6.4]{baker16}.
\end{proof}

\begin{remark}
Over $\B$, this lemma shows that if $M$ is the matroid associated to $L_w \subseteq \B^E$,
then $\pi_F(L_w)$ is the tropical linear space corresponding to the matroid $M\delete(E-F)$
and $\B^F \cap L_w$ is the tropical linear space corresponding to the matroid $M\contract(E-F)$.
See \cite[Theorem 4.1]{crowley17} for more details.
\end{remark}

\section{Linear extensions}

\subsection{Linear extensions and their underlying matroids}

In tropical geometry, graphs of regular functions on a balanced polyhedral complex are generally not balanced;
in particular, the graph of a linear function $\varphi \in (S^E)^\dual$ is typically not a tropical linear space.
However, there is a natural balanced polyhedral complex containing the set-theoretic graph.
Frenk, in his doctoral thesis, studied extensions of tropical linear spaces
over sub-semifields of $\T$, translating between polyhedral and algebraic definitions of this balanced polyhedral complex \cite[\S 4.2.2]{frenk13}. We generalize Frenk's construction, connecting it to matroidal notions and to the important operation of tropical modification. 
These methods allow us to construct an explicit example of a tropical linear space that is the tropical modification of a tropical linear space along a \textit{non-linear} rational function.  

\begin{definition}
An \emph{extension} of a tropical linear space $L_w \subseteq S^E$
is a tropical linear space $L_z \subseteq S^{E\cup P}$ for $P$ disjoint from $E$ such that $\pi_{E}(L_z) = L_w$.
An \emph{elementary} extension of $L_w \subseteq S^E$ is an extension $L_z \subseteq S^{E \cup P}$ such that $|P| = 1$.
\end{definition}

Because projection corresponds to the restricted matroid, the underlying matroid of an extension
is an extension of the underlying matroid.
Crapo studied elementary extensions of matroids, summarized in \cite[Chapter 7.2]{oxley92}.

For $\varphi \in (S^E)^\dual$, the tropical hyperplane defined by $\varphi + x_p \in (S^{E+p})^\dual$, 
which has tropical Pl\"ucker coordinates $\star_{E+p}(\varphi + x_p)$, 
contains the set-theoretic graph of $\varphi$.
This suggests defining the ``graph'' of $\varphi$ on a tropical linear space $L_w \subseteq S^E$ 
as the stable intersection of this tropical hyperplane with $L_w \oplus S$.

\begin{definition}
	Let $S$ be a totally ordered idempotent semifield.
	A \emph{linear extension} of a tropical linear space $L_w \subseteq S^E$
	is an elementary extension $L_z \subseteq S^{E + p}$ of $L_w$
	such that for some $\varphi \in (S^E)^\dual$, $z$ is of the form
	\begin{equation*}
		z= (w \wedge e_{p}) \cdot \star_{E+p}(\varphi + x_{p}) = \star_{E+p}(\star_E w \wedge (\varphi + x_p)).
	\end{equation*}
	Such an $L_z$ will be denoted $\linext{L_w}{\varphi}{p}$.
\end{definition}
By Lemma \ref{lemma: plucker-ops}, the vector $(w \wedge e_p) \cdot \star_{E+p}(\varphi + x_p)$ is always a tropical Pl\"ucker vector,
and for $I \subseteq E$ of size $\rank w$, $((w \wedge e_p) \cdot \star_{E+p}(\varphi + x_p))_I = w_I$.
Lemma \ref{lemma: minor-plucker-coordinates} thus shows that a linear extension is indeed an extension.

Given a matroid $M$ on ground set $E$, and given $X \subseteq E$,
the \emph{principal extension} of $M$ with respect to $X$,
denoted $\linext{M}{F}{p}$, can be described as the matroid on $E+p$ with independent sets
\begin{equation*}
	\{I \mid I \in \mathcal{I}(M)\} \cup \{I \cup p \mid I \in \mathcal{I}(M), \mathrm{cl}_M(I) \not\supseteq \mathrm{cl}_M(F)\},
\end{equation*}
where $\mathcal{I}(M)$ is the set of independent sets of $M$ \cite[Proposition 7.2.5]{oxley92}. 
\begin{lemma} \label{lemma: linear-extension-principal}
	Let $S$ be a totally ordered idempotent semifield, $L_w \subseteq S^E$ a tropical linear space with underlying matroid $M$, and $\varphi \in (S^E)^\dual$.
	If 
	\[F = \{i \in E \mid \varphi(e_i) \neq 0\},\] 
	then the underlying matroid of the linear extension $\linext{L_w}{\varphi}{p}$
	is the principal extension $\linext{M}{F}{p}$.
\end{lemma}
\begin{proof}
	The bases of $(w \wedge e_p) \cdot \star_{E+p}(\varphi + x_p)$
	are exactly $I$ when $I \subseteq E$ is a basis of $M$
	and $K \cup p$ when $\varphi(\sum_{i \in E - K} w_{K+i}e_i) \neq 0$.
	Hence, $K \cup p$ is a basis of $\linext{L_w}{\varphi}{p}$ if and only if the cocircuit of $M$ supported in $E- K$, i.e. $E - \mathrm{cl}_M(K)$, intersects $F$.
	Hence, the underlying matroid of $\linext{L_w}{\varphi}{p}$ is exactly the prinipcal extension $\linext{M}{F}{p}$.
\end{proof}

The following lemma shows that $\linext{L_w}{\varphi}{p}$ contains the set-theoretic graph of $\varphi$ on $L_w$.
\begin{lemma}\cite[Proposition 4.2.12]{frenk13} \label{lemma: linear-extension-set}
	Let $S$ be a totally ordered idempotent semifield,
	and $L_w \subseteq S^E$ be a tropical linear space.
	Let $\varphi \in (S^E)^\dual$.
	If $w \cdot \star \varphi \neq 0$, then 
	\begin{equation*}
		\linext{L_w}{\varphi}{p} = \{v + \varphi(v)e_p \mid v \in L_w\} \cup \{v+ae_p \mid v \in L_{w \cdot \star\varphi}, a \leq \varphi(v)\};
	\end{equation*}
	if $w \cdot \star \varphi = 0$, then 
	\begin{equation*}
		L_w +_\varphi p = \{v + \varphi(v)e_p \mid v \in L_w\}.
	\end{equation*}
\end{lemma}
\begin{proof}
	Let $z = (w \wedge e_p) \cdot \star_{E+p}(\varphi + x_p)$.
	First suppose that $w \cdot \star \varphi \neq 0$. 
	Let $L_w$ be rank $d$
	If $A \subseteq E+p$ is of size $d-1$ and contains $p$,
	then
	\begin{align*}
		\sum_{i \in (E+p) - A} z_{A+i}e_i &= \sum_{i \in (E+p)-A} (\star_E w \wedge (\varphi + x_p))_{(E+p)-A-i}e_i \\
		&= \sum_{i \in E - (A-p)} \star_{E}(\star_E w \wedge \varphi)_{A-p+i} e_i,
	\end{align*}
	which is the cocircuit of $L_{w \cdot \star \varphi}$
	associated to $A-p$.
	If $A \subseteq E$ is of size $d-1$,
	\begin{align*}
		\sum_{i \in (E+p)-A} z_{A+i}e_i &=  \sum_{i \in E -A} (\star_E w \wedge (\varphi + x_p))_{(E+p) -A-i}e_i \\
				&\qquad + (\star_E w \wedge (\varphi + x_p))_{E-A}e_p \\
		&= \sum_{i \in E - A} w_{A+i}e_i + \varphi\left(\sum_{i \in E-A} w_{A+i}e_i\right)e_p, 
	\end{align*}
	which is exactly  $v+ \varphi(v)e_p$ for $v$ the cocircuit of $L_w$ associated to $A$.
	If $w \cdot \star \varphi \neq 0$,
	the same arguments show that the cociruits of $z$ are exactly $v + \varphi(v)e_p$ for $v$ a cocircuit of $L_w$,
	and hence are included in the right-hand side.

	Now we prove the reverse inclusion. 
	If $v \in L_w$, then $v$ is a linear combination of cocircuits of $L_w$.
	Since $\varphi$ is linear, 
	$v+ \varphi(v)e_p$ is thus a linear combination of terms $v'+ \varphi(v')e_p$ where $v'$ is a cocircuit of $L_w$.
	Hence, $v + \varphi(v)e_p \in L_w +_\varphi p$.
	Now suppose that $w \cdot \star \varphi \neq 0$, $v \in L_{w \cdot\star\varphi}$, 
	and $a \leq \varphi(v)$.
	If $\varphi(v)= 0$, there is nothing to prove, so assume $\varphi(v) \neq 0$.
	Then 
	\begin{equation*}
		v+ ae_p = a\varphi(v)^{-1}(v + \varphi(v)e_p) + v
	\end{equation*}
	since addition is idempotent, showing that $v+ae_p \in L_w +_\varphi p$.
\end{proof}
A basic property of linear extensions is monotonicity:
\begin{lemma}\label{lemma: linear-extension-monotonic}
	Let $S$ be a totally ordered idempotent semifield, and 
	$L_w \subseteq L_z \subseteq S^E$ be tropical linear spaces.
	Let $\varphi \in (S^E)^\dual$.
	Then 
	\begin{equation*}
		\linext{L_w}{\varphi}{p} \subseteq \linext{L_z}{\varphi}{p}.
	\end{equation*}
\end{lemma}
\begin{proof}
	Let $\star$ denote the Hodge star on $S^{E +p}$.
	By Lemma \ref{lemma: wedge-is-monotonic},
	the wedge and its dual operation are both monotonic,
	so $L_{w \wedge e_p} \subseteq L_{z \wedge e_p}$
	and 
	$$L_{(w \wedge e_p) \cdot \star(\varphi + x_p)} \subseteq L_{(z \wedge e_p) \cdot \star(\varphi + x_p)},$$ 
	as desired.
\end{proof}

\subsection{Tropical modification and linearity}

The description of Lemma \ref{lemma: linear-extension-set} allows another tropical-geometric description of a linear extension as a \textit{tropical modification}.
Here we use tropical modifications in $\T^E$ in the sense of \cite{shaw13}.
The following Theorem is self-contained and will not be required for following sections in the paper.
\begin{theorem} \label{theorem: tropical-modification}
	Let $L \subseteq \T^E$ be a tropical linear space,
	and let $\varphi \in (\T^E)^\dual$ be a dual vector.
	Then the intersection of the linear extension $L +_\varphi p$ with the tropical torus $(\T^{E+p})^\times$ 
	is equal as a set to the tropical modification of $L \cap (\T^E)^\times$ along $\varphi$, which is equal to $\mathrm{div}_{L}(\varphi)$.
\end{theorem}
\begin{proof}
	In \cite[Theorem 4.11]{speyer08}, Speyer shows that if $w, w'$ are tropical Pl\"ucker vectors with all nonzero coordinates,
	then $L_{w\cdot w'} \cap (\T^\times)^E$ is the geometric stable intersection
	\[\lim\limits_{\tau \to (1,\ldots,1)} (L_w \cap \tau L_{w'})\cap (\T^\times)^E,\]
	where the limit is over an open subset of the torus $(\T^\times)^E = \R^E$.
	Speyer's proof also goes through when $w,w'$ are arbitrary tropical Pl\"ucker vectors, provided that $w \cdot w' \neq 0$.
	The tropical modification of $L$ along $\varphi$ in the torus
	is exactly the set-theoretic graph of $\varphi$ on $L$
	along with the undergraph on $\mathrm{div}_L(\varphi)$.
	By \cite[Proposition 2.12]{shaw13},
	$\mathrm{div}_L(\varphi)$ is exactly the geometric stable intersection of the hyperplane defined by $\varphi$
	and $L$ in the torus.
	This shows $(\linext{L}{\varphi}{p}) \cap (\T^\times)^E$ is the tropical modification of $L\cap (\T^\times)^E$ by $\varphi$.
%
\end{proof}

In \cite[Proposition 2.25]{shaw13}, Shaw proves that every rank-preserving nontrivial extension of a matroid
corresponds to a tropical modification of the corresponding tropical linear spaces.
However, tropical modifications by linear functions are principal extensions by Lemma \ref{lemma: linear-extension-principal}, and not every elementary extension of matroids is principal.
Hence, there are some elementary tropical modifications $L' \to L$
where $L$ and $L'$ are both tropical linear spaces,
but the rational function associated to the modification $L' \to L$ is \emph{not} linear.
\begin{example}
	Let $w = e_{123} + e_{124} + e_{134} + e_{234} \in \bigwedge^3 \B^4$ be the tropical Pl\"ucker vector corresponding to the uniform matroid $U_{3,4}$.
	Let $z = (e_1 + e_2)\wedge(e_3 + e_4)$; 
	then $L_z = \langle e_1 + e_2 \rangle \oplus \langle e_3 + e_4\rangle \subseteq L_w$. 	I claim that
	\begin{equation*}
		\mathrm{div}_{L_w}\left( \frac{(x_1 + x_2)(x_3 + x_4)}{x_1 + x_2 + x_3 + x_4}\right) = L_z.
	\end{equation*}
	At the level of algebraic cycles, we have
	\begin{align*}
		\mathrm{div}_{L_w}\left( \frac{(x_1 + x_2)(x_3 + x_4)}{x_1 +x_2 + x_3 + x_4}\right) 
		&= \mathrm{div}_{L_w}(x_1 + x_2)  + \mathrm{div}_{L_w}(x_3 + x_4) \\
		& \qquad- \mathrm{div}_{L_w}(x_1 + x_2 + x_3 + x_4).
	\end{align*}
	All of these divisors,
	of linear spaces by linear functions, 
	may be computed according to Theorem \ref{theorem: tropical-modification}. Their tropical Pl\"ucker vectors are
	\begin{align*}
		w_1 = w \cdot \star(x_1 + x_2) &= e_{13} + e_{14} + e_{23} + e_{24} + e_{34}, \\
		w_2 = w \cdot \star(x_3 + x_4) &= e_{12} + e_{13} +e_{14} + e_{23} +e_{24}, \\
		w_3 = w \cdot \star(x_1 + x_2 + x_3+x_4) &= e_{12} + e_{13} + e_{14} + e_{23} + e_{24} + e_{34},
	\end{align*}
	respectively.
	By computing the valuated cocircuits via Lemma \ref{lemma: generated-by-cocircuits}, and then taking the tropical span,
	it may be checked that the faces of $L_{w_i}$ (as a polyhedral complex)
	are given by the following:
	\begin{align*}
		L_{w_1} &=  \{(\alpha, \beta, \alpha, \alpha) \mid \alpha \geq \beta \}  \cup \{(\beta, \alpha,\alpha, \alpha) \mid \alpha\geq\beta\} \\
		& \qquad \cup \{(\alpha, \alpha, \beta, \beta) \mid \alpha \geq \beta \}; \\
		L_{w_2} &=\{(\alpha,\alpha,\beta,\alpha) \mid \alpha \geq \beta \} 
		\cup \{(\alpha,\alpha,\alpha,\beta) \mid \alpha \geq \beta\}\\ 
		& \qquad \cup \{(\beta, \beta, \alpha, \alpha) \mid \alpha \geq \beta\};\\   
		L_{w_3} &= \{(\beta,\alpha,\alpha,\alpha) \mid \alpha \geq \beta\} \cup \{(\alpha,\beta,\alpha,\alpha) \mid \alpha \geq \beta\} \\
		& \qquad \cup \{(\alpha,\alpha,\beta,\alpha) \mid \alpha \geq \beta\} \cup \{(\alpha,\alpha,\alpha,\beta) \mid \alpha \geq \beta\}. 
	\end{align*}
	Hence at the level of tropical cycles, we have
	\begin{equation*}
		L_{w_1} + L_{w_2} - L_{w_3} = \{(\alpha,\alpha,\beta,\beta)\mid \alpha,\beta \in \T\} = \langle e_1 + e_2\rangle \oplus \langle e_3 + e_4\rangle.
	\end{equation*}
	Thus, $L_z$ is a linear divisor of $L_w$, giving rise to a tropical modification $L_{w + z\wedge e_p} \to L_w$, associated to a non-linear rational function.
	In terms of the underlying matroids,
	the extension $L_{w + z \wedge e_p}$ of $L_w$ corresponds to an elementary extension of matroids 
	with the non-principal modular cut
	$\mathcal M = \{ \{1,2\},\{3,4\},\{1,2,3,4\}\}$.
\end{example}

\section{Tropical images}
\subsection{Tropical graphs and tropical images}

Extending our work in \S3, we now describe how to extend the graph, and thereby the image, of a linear function to a tropical linear space.

\begin{definition}
	Let $S$ be a totally ordered idempotent semifield, and $w \in \bigwedge^d S^E$ be a tropical Pl\"ucker vector.
	Let $A = \{a_{ij}\} \in S^{F \times E}$ be a matrix with columns indexed by $E$ and rows indexed by $F = \{f_1,\ldots, f_m\}$.
	Let $\{x_1,\ldots, x_n\}$ denote the dual basis to $S^E$ and $\{y_1,\ldots, y_m\}$ denote the dual basis to $S^F$.
	Let $\rho_j = \sum_{i \in E} a_{j i}x_i \in (S^E)^\dual$ be the form associated with the $j$th row of $A$.
	The \emph{tropical graph} of $A$ on $L_w$ is the tropical linear space in $S^{E \sqcup F}$ with tropical Pl\"ucker vector
	\begin{equation*}
		g(w,A) = (w \wedge f_F) \cdot \star(\rho_1 + y_1) \cdot \star(\rho_2 + y_2) \cdot \ldots \cdot \star(\rho_m + y_m) \in \bigwedge\nolimits^d S^{E \sqcup F},
	\end{equation*}
	where $f_F = f_1\wedge f_2 \wedge \cdots \wedge f_m$.
\end{definition}

From the definition, it follows that the tropical graph of $A$ on $L_w$ is the iterated linear extension
\begin{equation*}
	(\cdots ( \linext{(\linext{L_w}{\rho_1}{f_1})}{\rho_2}{f_2})  \cdots +_{\rho_m} f_m).
\end{equation*}
This tropical linear space does not depend on the order of the extensions, as the wedge product and its dual operation are commutative.

\begin{definition}
	Let $S$ be a totally ordered idempotent semifield, $A \in S^{F \times E}$ be a matrix, and $w \in \bigwedge^d S^E$ be 
	a tropical Pl\"ucker vector.
	The \emph{tropical image} of $L_w$ under $A$, denoted $\tropim{A}{L_w}$, is the projection 
	\begin{equation*}
		\pi_F\left(L_{g(w,a)}\right)
	\end{equation*}
	of the tropical graph of $A$ on $L_w$ onto the codomain of $A$.
\end{definition}
\begin{lemma} \label{lemma: image-basics}
	Let $S$ be a totally ordered idempotent semifield, $L_w \subseteq S^E$ a tropical linear space, and $A \in S^{F \times E}$ a matrix.
	Then
	\begin{enumerate}[(a)]
		\item{ $AL_w \subseteq \tropim{A}{L_w}$;}
		\item{ $\rank \tropim{A}{L_w} \leq \rank L_w$;}
		\item{ if $L_w \subseteq L_z \subseteq S^E$ is another tropical linear space,
			then
			\begin{equation*}
				\tropim{A}{L_w} \subseteq \tropim{A}{L_z}.
			\end{equation*}}
	\end{enumerate}
\end{lemma}
\begin{proof}
	(a) By Lemma \ref{lemma: linear-extension-set},
	the tropical graph of $A$ on $L_w$,
	equal to the linear extension $(\cdots ( \linext{(\linext{L_w}{\rho_1}{f_1})}{\rho_2}{f_2}) \cdots +_{\rho_m}f_m)$,
	contains $v + \sum_{i=1}^m \rho_i(v)f_i$ for all $v \in L_w$.
	Hence, the tropical image contains $\sum_{i=1}^m\rho_i(v)f_i = Av$ for all $v \in L_w$.

	(b) By construction, the tropical graph of $A$ on $L_w$ has the same rank as $L_w$,
	and projecting onto a coordinate subspace does not increase rank.

	(c) This follows from Lemma \ref{lemma: linear-extension-monotonic}
	and that the coordinate projection $\pi_F$ is monotonic.
\end{proof}

The following lemma gives a more explicit formulation of the Pl\"ucker coordinates of the tropical image.
\begin{lemma} \label{lemma: plucker-image-tropdet}
	Let $S$ be a totally ordered idempotent semifield,
	$w \in \bigwedge^d S^E$ be a tropical Pl\"ucker vector,	
	and $A \in S^{F \times E}$ be a matrix.
	The tropical Pl\"ucker coordinates $z$ of $\tropim{A}{L_w}$ are 
	\begin{equation*}
		z_J = \sum_{I \subseteq E - K} \tropdet{A_{JI}}w_{I \cup K},
	\end{equation*}
	where $K$ is any basis for the contraction of the tropical graph of $A$ on $L_w$ to $E$, 
	and $\tropdet{A_{JI}}$ is the tropical $J \times I$ minor of $A$.
\end{lemma}
\begin{proof}
	Let $g(w,A)$ denote the tropical Pl\"ucker vector of the tropical graph of $A$ on $L_w$.
	By Lemma \ref{lemma: minor-plucker-coordinates}, 
	if $K$ is a basis for the contraction of $g(w,A)$ to $E$,
	then $J \mapsto g(w,A)_{K \cup J}$ for $J \subseteq F$ are tropical Pl\"ucker coordinates for $\tropim{A}{L_w}$.
	By definition,
	if $\star$ is the Hodge star on $S^{E \sqcup F}$ 
	and $\rho_j$ denotes the form associated to the $j$th row of $A$,
	\begin{align*}
		g(w,A)_{K \cup J} 	&=\left[(w \wedge f_F) \cdot \star(\rho_1 + y_1) \cdot \ldots \cdot \star(\rho_m + y_m)\right]_{K \cup J} \\
				&=\left[(\star_E w) \wedge \bigwedge_{j=1}^m (\rho_j + y_j)\right]_{(E-K) \cup (F-J)} 
	\end{align*}
	by duality. Now since $\star_E w$ is supported in $E$
	and $\rho_j+y_j$ is supported in $E \cup j$ for all $j \in F$,
	the only terms that contribute to $y_{F-J}$
	in the above wedge product are $\rho_j + y_j$ for $j \in F-j$.
	Hence,
	\begin{align*}
		g(w,A)_{K\cup J} 	&= \left[(\star_E w) \wedge \bigwedge_{j \in J} (\rho_j+y_j)\right]_{E-K} \\
				&= \sum_{I\subseteq E-K} (\star_Ew)_{E-K-I} (\bigwedge_{j \in J}\rho_j)_I \\
				&= \sum_{I \subseteq E-K} w_{I\cup K} \tropdet{A_{JI}},
	\end{align*}
	as desired.
\end{proof}
This lemma shows that if $A\in S^{F \times E}$ has a nonzero maximal minor,
then $\tropim{A}{S^E}$ is equal to
the Stiefel tropical linear space associated to $A^T$ (see \cite{fink15}).
Even if $A$ has no nonzero maximal minor, we have
\begin{corollary}
	For arbitrary $A \in S^{F\times E}$,
	if $B$ is a submatrix of $A$ formed by the columns of a maximal non-zero minor of $A$,
	then 
	\[\tropim{A}{S^E} = \tropim{B}{S^E},\] 
	and hence is a Stiefel tropical linear space.
	In particular, $\tropim{A}{S^E}$ has rank equal to the size of the largest nonzero minor of $A$.
\end{corollary}

\subsection{The underlying matroid of the tropical image}

In this section, we compute the underlying matroid of the tropical image,
which turns out to be a classical matroid-theoretic construction.
Given a matroid $M$ on $E$ and a bipartite graph $\Gamma$ on $E \sqcup F$,
the \textit{induced matroid} $\Gamma(M)$ is the matroid on $F$
whose independent sets are exactly the $J\subseteq F$ that have a perfect matching in $\Gamma$ to an independent set of $M$
(see \cite[Chapter 12.2]{oxley92}).
\begin{theorem} \label{theorem: image-underlying-matroid}
	Suppose $M$ is a matroid on $E$ corresponding to the tropical linear space $L_w \subseteq \B^E$.
	Let $\Gamma$ be a bipartite graph on $E \sqcup F$, with incidence matrix $A \in \B^{F \times E}$.
	Then the induced matroid $\Gamma(M)$ corresponds to the tropical linear space $\tropim{A}{L_w} \subseteq \B^F$.
\end{theorem}
\begin{proof}
	Non-zero terms in the expansion of $\tropdet{A_{JI}}$ correspond to bijections $f: J \to I$
	such that $A_{j,f(j)} = 1$ for all $j \in J$, i.e. $j$ and $f(j)$ are adjacent in $\Gamma$ for all $j \in J$.
	Hence, $\tropdet{A_{JI}} = 1$ if and only if there is a matching from $J$ to $I$ in $\Gamma$, and is $0$ otherwise.
	By Lemma \ref{lemma: plucker-image-tropdet}, the tropical Pl\"ucker coordinates of $\tropim{A}{L_w}$ are 
	\begin{equation}\label{equation: underlying-coordinates}
		z_J = \sum_{I \subseteq E-K} \tropdet{A_{JI}}w_{K\cup I},
	\end{equation}
	where $K$ is a basis for the contraction of the tropical graph to $E$.
	By Lemma \ref{lemma: minor-plucker-coordinates}, 
	$K$ is a minimal subset of $E$
	such that \eqref{equation: underlying-coordinates}
	does not vanish for all $J$, and hence the bases of $\tropim{A}{L_w}$ are the bases of $\Gamma(M)$.
\end{proof}
\begin{corollary} \label{corollary: transversal-matroids}
	The tropical linear spaces assocaited to transversal matroids are exactly the tropical images of free $\B$-modules.
\end{corollary}
\begin{proof}
	Transversal matroids are the induced matroids of free matroids.
\end{proof}
Corollary \ref{corollary: transversal-matroids} also follows from the description of Stiefel tropical linear spaces as the tropical image of free modules.
\begin{remark}
On \cite[p. 108]{frenk13}, Frenk investigates the \emph{valuated linking system}
associated to a bipartite graph with edges weighted by elements of a semifield.
If $A$ is the weighted incidence matrix of the weighted bipartite graph,
Frenk's definition assigns to the pair $(I,J)$ the weight $\tropdet{A_{JI}}$.
Thus, the tropical image under $A$ is, in Frenk's language, 
the image under the valuated linking system associated to the weighted graph of $A$.
\end{remark}
The tropical image is not functorial with respect to matrix multiplication.
It may be the case that $\tropim{B}{\tropim{A}{L_w}}$ is not equal to $\tropim{C}{L_w}$ for any matrix $C$!
\begin{example}\label{example: no-composition}
	Consider the $\B$-matrices 
	\begin{align*}
		B = \begin{bmatrix}
			1&0&0\\1&0&0\\
			0&1&0\\0&1&0\\
			0&0&1\\0&0&1 
		\end{bmatrix}	&&
		A = \begin{bmatrix}
			1&1\\1&1\\1&1\end{bmatrix}
	\end{align*}
	The tropical image $\tropim{B}{\tropim{A}{\B^3}}$ is exactly the rank 2 truncation of the transversal matroid $\tropim{B}{\B^3}$.
	This truncation has three cyclic flats of rank 1, namely $\{1,2\}$, $\{3,4\}$, and $\{5,6\}$. 
	By a result of Brylawski, a rank $r$ transversal matroid has at most $\binom{r}{k}$ rank-$k$ cyclic flats \cite{brylawski75}.
	Thus, $\tropim{B}{\tropim{A}{\B^3}}$ is not transversal, and hence is not of the form $\tropim{C}{\B^3}$ for any matrix $C$.
	\begin{figure}[ht]
	\begin{tikzpicture}
		\pgfmathsetmacro{\x}{2*asin(1/20)}
		
		\draw (0:1) -- (120:1) -- (240:1) -- cycle;
		\fill 
		({-\x}:1) 	circle(0.1) 
		({\x}:1) 	circle(0.1)
		({120-\x}:1) 	circle(0.1)
		({120+\x}:1) 	circle(0.1) 
		({240-\x}:1) 	circle(0.1)
		({240+\x}:1) 	circle(0.1);
		\node () at 	(-\x:1.3) 	{1};
		\node () at 	(\x:1.3) 	{2};
		\node () at 	({120-\x}:1.3) 	{3};
		\node () at 	({120+\x}:1.3) 	{4};
		\node () at 	({240-\x}:1.3) 	{5};
		\node () at 	({240+\x}:1.3) 	{6};
		\draw (4,0) -- (6,0);
		\foreach \i in {4,5,6}
		{
		\foreach \j in {-0.1,0.1}
			\fill (\i, \j) circle(0.1);
		}
		\node ( ) at (4,0.4) 	{1};
		\node ( ) at (4,-0.4) 	{2};
		\node ( ) at (5,0.4) 	{3};
		\node ( ) at (5,-0.4) 	{4};
		\node ( ) at (6,0.4) 	{5};
		\node ( ) at (6,-0.4) 	{6};
	\end{tikzpicture}
	\caption{Geometric representations of the matroids $\tropim{B}{\B^3}$ and $\tropim{B}{\tropim{A}{\B^3}}$ of Example \ref{example: no-composition}.}
	\end{figure}
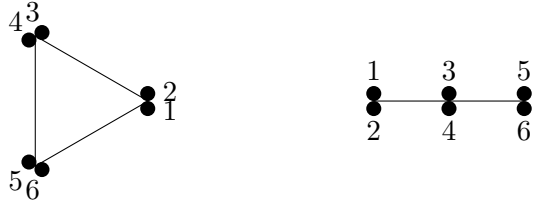
\end{example}

\subsection{Realizability}
Recall that if $k$ is a field, $S$ a totally ordered idempotent semifield, and $\val: k \to S$ is a surjective valuation,
then we may \textit{tropicalize} a linear subspace $\Lambda \subseteq k^n$ by taking the valuation of every element of $\Lambda$ coordinate-wise. 
The result is known as the \textit{tropicalization} of $\Lambda$ and will be denoted $\trop{\Lambda}$.
The tropical Pl\"ucker coordinates of $\trop{\Lambda}$ are the image of the Pl\"ucker coordinates of $\Lambda$ under $\val$ \cite{speyer04}.
A tropical linear space is \textit{realizable} with respect to a particular valuation if it is a tropicalization of some linear space under that valuation.
For matroids, a realizable tropical linear space via $k \to \B$ is exactly a representable matroid over $k$.
From the perspective of matroids over hyperfields, a valuation $\val: k \to S$ induces a hyperfield homomorphism from $k$ to the canonical hyperfield associated to $S$, and tropicalization is exactly the pushforward under this hyperfield homomorphism \cite{baker16}.

Let $M$ be a matroid on $E$ and $\Gamma$ a bipartite graph on $E \sqcup F$.
Piff and Welsh proved in 1970 that if $M$ is representable over a sufficiently large field, 
then $\Gamma(M)$ is representable over that field as well \cite[Proposition 12.2.16]{oxley92}.
Fink and Rinc\'on also observed that a Stiefel tropical linear space in $\T^n$ 
is the tropicalization of the image of a general lift of its matrix.
In this section, we generalize by investigating the realizability of the tropical image.

Given a field $\kappa$ and a totally ordered semifield $S$, let $\puiseux{S}$ 
denote the field of generalized power series $\sum_{s \in T} \alpha_st^s$, where $T \subseteq S^\times$ is well-ordered and $\alpha_s \in \kappa$ for all $s \in T$.
There is a surjective valuation $\val: \puiseux{S} \to S$ 
that maps each non-zero power series to the inverse (in $S$) of its degree (here the inverse is to accord with our convention of maximum for addition).
Poonen showed that every valuation that is trivial on its initial field factors through such a valuation \cite{poonen93}.

The main technical tool is Lemma \ref{lemma: lift}, which describes when solutions to certain valuated equations may be lifted.
Proving the existence of solutions requires the following lemma on the non-vanishing of polynomials.
\begin{lemma}{\cite[Lemma 2]{lindstroem73}} \label{lemma: poly-non-zero}
	Let $\kappa$ be a field, and let $f \in \kappa[x_1,\ldots, x_n]$ be a polynomial.
	Suppose the degree of $f$ in each $x_i$ alone is less than $|\kappa|$.
	Then there exists $\alpha \in \kappa^n$ such that $f(\alpha) \neq 0$.
\end{lemma}

\begin{lemma} \label{lemma: lift}
	Let $S$ be a totally ordered idempotent semifield, $\kappa$ a field,
	and let $\val: \puiseux{S} \to S$ be the standard valuation.
	Let $f_1,f_2,\ldots, f_m$ be polynomials in
	$\puiseux{S}[x_1,\ldots, x_n]$,
	and let $a \in S^n$.
	If $d_i$ is the maximum degree of $f_i$ in each $x_j$ alone,
	and $|\kappa| > \sum_{i=1}^m d_i$, 
	then there exists a non-empty Zariski open subset of $\kappa^n$
	such that if the leading coefficients of $\alpha \in \val^{-1}(a)$ lie in that subset, then
	\begin{equation*}
		\val(f_i)(a) = \val(f_i(\alpha))
	\end{equation*}	
	for all $i$. 
\end{lemma}

\begin{proof}
	Suppose $f_i = \sum_{\mathbf{u}} c^i_\mathbf{u} \mathbf{x}^\mathbf{u}$. 
	Collect all terms $c^i_\mathbf{u}$ such that $\val(f_i)(a) = \val(c^i_\mathbf{u})a_1^{u_1}\cdots a_n^{u_n}$
	into a new polynomial $g_i$.
	Then $\alpha \in \val^{-1}(a)$ satisfies $\val(f_i)(a) = \val(f_i(\alpha))$
	if and only if $g_i$ does not vanish on the leading coefficients of $\alpha$.
	The degree of $g_i$ in a single indeterminate is at most $d_i$,
	so by Lemma \ref{lemma: poly-non-zero},
	there is a point $\alpha$ in $\kappa^n$ such that $g_1g_2\cdots g_m$ does not vanish,
	i.e.\ where no $g_i$ vanishes.
\end{proof}

\begin{theorem} \label{theorem: realizable}
	Let $S$ be a totally ordered idempotent semifield, 
	and $\val: \puiseux{S} \to S$ be the standard valuation.
	Let $\Lambda \subseteq \puiseux{S}^E$ be a rank $d$ linear subspace,
	and $L = \trop{\Lambda}$.
	Let $A \in S^{F \times E}$ be a matrix.
	If $|\kappa| > \binom{|E|+|F|}{d}$,
	then $\tropim{A}{L}$ is realizable over $\puiseux{S}$;
	indeed for $\Delta \in \val^{-1}(A)$ with generic leading coefficients we have
	\begin{equation*}
		\tropim{A}{L} = \trop{\Delta \Lambda}.
	\end{equation*}
\end{theorem}
\begin{proof}
	Because coordinate projection commutes with tropicalization,
	it suffices to show for $\Delta \in \val^{-1}(A)$ with generic leading coefficients
	that the tropicalization of the graph of $\Delta$ on $\Lambda$
	is the tropical graph of $A$ on $L$.

	Let the standard basis of $\puiseux{S}^{E \cup F}$ be $e_1,\ldots, e_n, f_1,\ldots,f_m$
	with dual basis $x_1,\ldots,x_n,y_1,\ldots,y_m$.
	Let $p$ be the Pl\"ucker vector for $\Lambda$, defined up to a scalar.
	Then the (classical) graph of a matrix $X$ with $X_{ij}$ on $\Lambda \subseteq \puiseux{S}$ 
	has Pl\"ucker vector
	\begin{equation} \label{equation: classical-graph-coordinates}
		\star_{E \sqcup F} \left( \star_E p \wedge \bigwedge_{j=1}^m \left(y_j - \sum_{i=1}^n X_{ji} x_i \right)\right),
	\end{equation}
	where $\star_{E\sqcup F}$ and $\star_E$ denote the classical Hodge stars on $\puiseux{S}^{E\sqcup F}$ and $\puiseux{S}^{E}$.
	For each $I \in \binom{E \sqcup F}{d}$, let $f_I \in \puiseux{S}[X_{ji}]$ be the polynomial in indeterminates indexed by $F \times E$
	that is the coefficient of $e_I$ in \eqref{equation: classical-graph-coordinates}.
	The polynomial $f_I$ is only well-defined up to sign,
	but $\val(f_I)$ and $\val(f_I(\Delta))$ are always well-defined for $\Delta \in \puiseux{S}^{F \times E}$. 
	We calculate the coefficients of $f_I$.
	Let $J = I \cap E$ and $K = I \cap F$.
	Then the coefficient of $e_{I}$ in \eqref{equation: classical-graph-coordinates} is
	\begin{align*}
		\pm \left(\star_E p \wedge \bigwedge_{j=1}^m \left(y_j - \sum_{i=1}^n X_{ji} x_i\right)\right)_{E-J \cup F - K} \\
		 = \pm \left( \star_E p \wedge \bigwedge_{j \in K} \left(- \sum_{i=1}^n X_{ji} x_i\right)\right)_{E-J} \\
	\end{align*}
	and hence the monomials appearing in $f_I$ are of the form $\prod_{j \in K} X_{j,\sigma(j)}$ for some injective $\sigma: K \to (E-J)$.
	The coefficient of this monomial is exactly $\pm p_{J\cup\sigma(K)}$.
	Hence, every coefficient of a monomial in $f_I$ is of the form $\pm p_{J'}$ for some $J' \subseteq E$.
	
	By definition of the classical and tropical exterior algebras,
	$I \mapsto \val(f_I)(A)$ are the tropical Pl\"ucker coordinates for the tropical graph of $A$ on $L$. 
	The degree of $f_I$ in each indeterminate is at most $1$, and by hypothesis 
	$$|\kappa| > \left| \left\{f_I: I \subseteq E\sqcup F, |I|=d\right\}\right|.$$
	By Lemma \ref{lemma: lift}, there is a non-empty Zariski open subset of $\kappa^{F \times E}$
	such that if the leading coefficients of $\Delta \in \val^{-1}(A)$ lie in that open subset,
	then the tropical graph of $A$ on $L$ is the tropicalization of the graph of $\Delta$ on $\Lambda$.
\end{proof}

\section{Stable sum}

The tropical image allows us to understand the stable sum of tropical linear spaces as a tropical image. 
This also provides a generalization of the stable sum to tropical linear spaces $L_w, L_z$ where $w \wedge z = 0$.

The tropical addition map $+: S^{E \sqcup E} \to S^E$, coinciding with the tropicalization of classical addition, has matrix
\[ A_+ = \begin{bmatrix} I_{E} & I_{E} \end{bmatrix},\]
where $I_{E}$ is the identity matrix on $S^E$. 
The key observation is that the underlying bipartite graph of this matrix 
is the same graph used to define the matroid union \cite[Theorem 12.3.1]{oxley92},
suggesting investigating the tropical image under $A_+$ over an arbitrary semifield.
\begin{theorem} \label{theorem: stable-sum}
	Let $S$ be a totally ordered idempotent semifield.
	Let $L_w$ and $L_z$ be tropical linear spaces in $S^E$.
	If $w \wedge z \neq 0$, then
	\begin{equation*}
		\tropim{A_+}{L_w \oplus L_z} = L_{w \wedge z},
	\end{equation*}
	where $A_+$ is the matrix of the addition map $+: S^{E \sqcup E} \to S^E$.
\end{theorem}
\begin{proof}
	Let us write the domain as $S^{E' \sqcup E''}$ 
	for $E' = \{e_1',\ldots,e_n'\}$ and $E'' = \{e_1'', \ldots, e_n''\}$
	two copies of $E$
	The underlying bipartite graph of addition 
	is the graph on $(E' \sqcup E'') \sqcup E$ 
	where a vertex in $E$ is exactly adjacent to its two copies in $E' \sqcup E''$.
	By Theorem \ref{theorem: image-underlying-matroid} and \cite[Theorem 12.3.1]{oxley92},
	the underlying matroid of $\tropim{A_+}{L_w \oplus L_z}$ is the matroid union
	of the underlying matroids of $L_w$ and $L_z$.
	Because $w \wedge z \neq 0$, these matroids have disjoint bases,
	so $\rank \tropim{+}{L_w \oplus L_z} = \rank L_w + \rank L_z$.

	Now we calculate the tropical Pl\"ucker coordinates of $\tropim{+}{L_w \oplus L_z}$.
	Since the ranks of $L_w \oplus L_z$ and its tropical image are the same,
	by Lemma \ref{lemma: plucker-image-tropdet}, the tropical Pl\"ucker coordinates are given by
	\begin{align*}
		J \mapsto \sum_{I \subseteq E' \sqcup E''} \tropdet{(A_+)_{JI}} w_{I \cap E'}z_{I \cap E''}
		= \sum_{J = J' \sqcup J''} w_{J'} z_{J''}
	\end{align*}
	since $\tropdet{(A_+)_{JI}}$ is $1$ if and only if the copies of $I \cap E'$ and $I \cap E''$ in $E$ are a partition of $J$.
	This shows that the tropical Plucker vector of the tropical image is exactly $w \wedge z$, as desired.
\end{proof}

The stable sum (and dually, stable intersection) of tropical linear spaces 
$L_w \subseteq S^E$ and $L_z \subseteq S^E$ has so far only been defined when $w\wedge z \neq 0$ \cite{fink15}.
Theorem \ref{theorem: stable-sum} suggests the following definition for arbitrary $L_w$ and $L_z$:
\begin{definition}
	Let $L_w \subseteq S^E$ and $L_z \subseteq S^E$ be tropical linear spaces.
	Then the \textit{stable sum} of $L_w$ and $L_z$ is 
	\begin{equation*}
		L_w \stsum L_z = \tropim{A_+}{L_w \oplus L_z},
	\end{equation*}
\end{definition}
In the case that $w \wedge z$ is possibly zero, the stable sum is equal to a rank-additive stable sum of subspaces:
\begin{corollary}
	Let $L_w $ and $L_z$ be tropical linear spaces in $S^E$.
	Then if $L_{w'}$ and $L_{z'}$ are subspaces of $L_w$ and $L_z$
	such that $w'\wedge z'$ is non-zero and has rank equal to $L_w \stsum L_z$,
	then 
	\begin{equation*}
		L_{w' \wedge z'} = L_w \stsum L_z.
	\end{equation*}
\end{corollary}
\begin{proof}
	Since $L_{w'}\oplus L_{z'} \subseteq L_w \oplus L_z$,
	by Lemma \ref{lemma: image-basics},
	\begin{equation} \label{equation: sub-stable-sum}
		L_{w'} \stsum L_{z'} \subseteq L_w \stsum L_z.
	\end{equation}
	But by Theorem \ref{theorem: stable-sum},
	$L_{w'} \stsum L_{z'} = L_{w' \wedge z'}$,
	and so the tropical linear spaces in \eqref{equation: sub-stable-sum} have the same rank.
	Hence, they are equal.
\end{proof}
	
Combining Theorem \ref{theorem: realizable} and \ref{theorem: stable-sum} shows that
if $\Lambda_1$ and $\Lambda_2$ are transverse subspaces of $\puiseux{S}^E$,
\begin{equation*}
	\trop{\Lambda_1} \stsum \trop{\Lambda_2} = \trop{\gamma_1\Lambda_1 + \gamma_2\Lambda_2}
\end{equation*}
for generic $(\gamma_1, \gamma_2) \in (\kappa^\times)^{E \sqcup E}$. 
Dualizing and observing that $\trop{\gamma\Lambda} = \trop{\Lambda}$ for any $\gamma \in (\kappa^\times)^E$
shows that if $\Lambda_1$ and $\Lambda_2$ intersect transversely, then 
\begin{equation*}
	\trop{\Lambda_1} \stint \trop{\Lambda_2} = \trop{\Lambda_1 \cap \gamma\Lambda_2}
\end{equation*}
for generic $\gamma \in (\kappa^\times)^E$.
This result is known for tropical varieties of all degrees \cite[Theorem 3.6.1]{maclagan15},
but the proof is entirely in the language of balanced polyhedral complexes.
The methods of this paper provide a new and algebraic proof in the linear case.
\\

\bibliography{references}

\begin{thebibliography}{10}

\bibitem{anderson18}
Laura Anderson.
\newblock Vectors of matroids over tracts.
\newblock {\em arXiv preprint arXiv:1607.04868v3}, 2018.

\bibitem{baker16}
Matthew Baker and Nathan Bowler.
\newblock Matroids over hyperfields.
\newblock {\em arXiv preprint arXiv:1601.01204}, 2016.

\bibitem{brylawski75}
Thomas~H. Brylawski.
\newblock An affine representation for transversal geometries.
\newblock {\em Studies in Appl. Math.}, 54(2):143--160, 1975.

\bibitem{crowley17}
Colin Crowley, Noah Giansiracusa, and Joshua Mundinger.
\newblock A module-theoretic approach to matroids.
\newblock {\em arXiv preprint arXiv:1712.03440}, 2017.

\bibitem{dress92}
Andreas W.~M. Dress and Walter Wenzel.
\newblock Valuated matroids.
\newblock {\em Adv. {M}ath.}, 93(2):214--250, 1992.

\bibitem{fink15}
Alex Fink and Felipe Rinc{\'o}n.
\newblock Stiefel tropical linear spaces.
\newblock {\em J. Combin. Theory Ser. A}, 135:291--331, 2015.

\bibitem{frenk13}
Bart Frenk.
\newblock {\em Tropical varieties, maps and gossip}.
\newblock PhD thesis, Eindhoven University of Technology, 2013.

\bibitem{giansiracusa16}
Jeffrey Giansiracusa and Noah Giansiracusa.
\newblock Equations of tropical varieties.
\newblock {\em Duke Math. J.}, 165(18):3379--3433, 2016.

\bibitem{giansiracusa17}
Jeffrey Giansiracusa and Noah Giansiracusa.
\newblock A {G}rassmann algebra for matroids.
\newblock {\em Manuscripta Math.}, 156(1-2):187--213, 2018.

\bibitem{haque12}
Mohammad~Moinul Haque.
\newblock Tropical incidence relations, polytopes, and concordant matroids.
\newblock {\em arXiv preprint arXiv:1211.2841}, 2012.

\bibitem{lindstroem73}
Bernt Lindstr\"om.
\newblock On the vector representations of induced matroids.
\newblock {\em Bull. London Math. Soc.}, 5:85--90, 1973.

\bibitem{maclagan15}
Diane Maclagan and Bernd Sturmfels.
\newblock {\em Introduction to tropical geometry}, volume 161 of {\em Graduate
  Studies in Mathematics}.
\newblock American Mathematical Society, 2015.

\bibitem{murota96}
Kazuo Murota.
\newblock On exchange axioms for valuated matroids and valuated delta-matroids.
\newblock {\em Combinatorica}, 16(4):591--596, 1996.

\bibitem{oxley92}
James Oxley.
\newblock {\em Matroid theory}, volume~21 of {\em Oxford Graduate Texts in
  Mathematics}.
\newblock Oxford University Press, second edition, 2011.

\bibitem{poonen93}
Bjorn Poonen.
\newblock Maximally complete fields.
\newblock {\em Enseign. Math. (2)}, 39(1-2):87--106, 1993.

\bibitem{shaw13}
Kristin~M. Shaw.
\newblock A tropical intersection product in matroidal fans.
\newblock {\em SIAM J. Discrete Math.}, 27(1):459--491, 2013.

\bibitem{speyer04}
David Speyer and Bernd Sturmfels.
\newblock The tropical {G}rassmannian.
\newblock {\em Adv. Geom.}, 4(3):389--411, 2004.

\bibitem{speyer08}
David~E. Speyer.
\newblock Tropical linear spaces.
\newblock {\em SIAM J. Discrete Math.}, 22(4):1527--1558, 2008.

\end{thebibliography}
\bibliographystyle{plain}

\end{document}